\documentclass[12pt]{extarticle}
\usepackage{amsmath, amsthm, amssymb, mathtools, hyperref, color}
\usepackage[shortlabels]{enumitem}
\usepackage{graphicx}
\usepackage{adjustbox}
\usepackage{xspace}
\usepackage{cleveref}
\usepackage{tikz}
\usepackage{multirow}
\usepackage{booktabs}
\usepackage{array}
\usepackage{tabularx}
\usetikzlibrary{arrows,calc,fit,matrix,positioning,cd}
\usepackage{float}
\usepackage{natbib}

\definecolor{cb-yellow}{RGB}{221,170,51}
\definecolor{cb-red} {RGB}{187,85,102}
\definecolor{cb-green}{RGB}{17,119,51}

\hypersetup{
    pdfauthor={Carlos Am\'endola, Janike Oldekop, Maximilian Wiesmann},
    pdftitle={The ML degree of Toric Models is Monotonic},
    pdfsubject={},
    pdfkeywords={},
    colorlinks=true,
    linkcolor=cb-red,
    filecolor=cb-red,      
    urlcolor=cb-green,
    citecolor=cb-green,
}

\oddsidemargin \evensidemargin
\newtheorem{theorem}{Theorem}
\newtheorem*{theorem*}{Theorem}
\newtheorem{corollary}[theorem]{Corollary}

\newtheorem{problem}[theorem]{Problem}

\theoremstyle{definition}
\newtheorem{definition}[theorem]{Definition}

\newtheorem{example}[theorem]{Example}
\newtheorem{defprop}[theorem]{Definition/Proposition}

\newcommand{\RR}{\mathbb{R}}
\newcommand{\QQ}{\mathbb{Q}}
\newcommand{\ZZ}{\mathbb{Z}}
\newcommand{\NN}{\mathbb{N}}
\newcommand{\PP}{\mathbb{P}}
\newcommand{\CC}{\mathbb{C}}

\newcommand{\puiseux}[1]{{#1}\{\!\{t\}\!\}}

\newcommand{\cH}{\mathcal{H}}
\newcommand{\cL}{\mathcal{L}}
\newcommand{\M}{\mathcal{M}}
\newcommand{\cN}{\mathcal{N}}

\newcommand{\V}{\mathcal{V}}

\DeclareMathOperator{\conv}{conv}
\DeclareMathOperator{\mldeg}{MLdeg}
\DeclareMathOperator{\im}{im}

\newcolumntype{P}[1]{>{\centering\arraybackslash}p{#1}}
\date{}

\title{\textbf{The Maximum Likelihood Degree of \\ Toric Models is Monotonic}}
\author{Carlos Am\'endola, Janike Oldekop and Maximilian Wiesmann}

\begin{document}
\maketitle

\begin{abstract}
    We settle a conjecture by Coons and Sullivant stating that the maximum likelihood (ML) degree of a facial submodel of a toric model is at most the ML degree of the model itself. We discuss the impact on the ML degree from observing zeros in the data. Moreover, we connect this problem to tropical likelihood degenerations, and show how the results can be applied to discrete graphical and quasi-independence models.
\end{abstract}

\section{Introduction}
\label{sec:intro}

In algebraic statistics \cite{AlgebraicStatistics}, a statistical model is represented by an algebraic variety. The class of \emph{log-affine models}, also known as discrete regular exponential families, corresponds to \emph{scaled projective toric varieties} $X_{A,c}$ \cite{amendola2019maximum}. Here, $A$ encodes a polytope associated to the toric variety and $c$ are scaling parameters. A key notion of likelihood geometry \cite{LikelihoodGeometry} is the \emph{maximum likelihood (ML) degree} \cite{catanese2006maximum}. It is an algebraic invariant of the variety representing a statistical model capturing the algebraic complexity of ML estimation. \par 

The ML degree makes an important reappearance in physics in the context of \emph{scattering amplitudes}. Here, the ML degree counts the number of solutions to the scattering equations \cite{sturmfels2021likelihood}. Most recently, it was shown that the ML degree of a scaled toric variety determines the number of summands appearing in the \emph{toric amplitude} \cite{telen2025toric}.

The ML degree of a scaled toric variety $X_{A,c}$ depends both on the polytope $\conv(A)$ and the scalings $c$. For generic $c$, the ML degree simply equals the normalized volume of $\conv(A)$, which is the degree of $X_{A,c}$. However, the ML degree can \emph{drop}. This is precisely the case if $c$ lies on the principal $A$-determinant \cite{DiscriminantsResultantsAndMultidimensionalDeterminants, amendola2019maximum}. The complete dependence of the ML degree on the scaling is captured by an Euler stratification \cite{EulerStrati}. \par 

Further research in this direction studies the ML degree of \emph{facial submodels}. Given a face~$F$ of the polytope $\conv(A)$, the facial submodel for $F$ of $X_{A,c}$ corresponds to the scaled toric variety $X_{A_F,c_F}$, where the subscript denotes restriction to the lattice points contained in the face $F$. Often, if a combinatorial structure (e.g.\ a graph) underlies a statistical model, the model described by a substructure is then a facial submodel. This is the case, for example, for discrete graphical models \cite{geiger2006toric} or quasi-independence models \cite{QuasiIndependenceModelsWithRationalMaximumLikelihoodEstimator}. It is proved in \cite[Theorem 3.2]{QuasiIndependenceModelsWithRationalMaximumLikelihoodEstimator} that if a toric variety has ML degree one, then all facial submodels need to have ML degree one. Moreover, the authors conjecture that the ML degree of a facial submodel cannot increase, i.e.\ the ML degree is \emph{monotonic} with respect to the face poset \cite[Conjecture 3.5]{QuasiIndependenceModelsWithRationalMaximumLikelihoodEstimator}. The main goal of the present article is to prove this statement.

\begin{theorem*}[Monotonicity of toric ML degree]
    Let $X_{A,c}$ be a scaled toric variety and let $F$ be a face of $\conv(A)$. Then the ML degree is monotonic on the facial submodels $X_{A_F,c_F}$ of $X_{A,c}$, that is,
    \[
        \mldeg(X_{A_F,c_F}) \leq \mldeg(X_{A,c}).
    \]
\end{theorem*}

For (non-toric) Gaussian graphical models, an analogous statement, monotonicity of the ML degree with respect to subgraphs, has been recently shown in \cite[Corollary 2.2]{amendola2024maximum}. \par 

\begin{example}
    \label{ex:2-dilated-cube}
    Consider the three-dimensional cube with side-length two, located in the positive orthant with one vertex being the origin, see Figure \ref{fig:polytopes} (a). Let $A$ be the matrix whose columns are the lattice points, and let $c$ be the all-ones vector of length~27. The variety $X_{A,c}$ is a Segre--Veronese embedding of $\PP^1\times \PP^1\times \PP^1 \subset \PP^{26}$. In statistical terms, it represents the independence model of three random variables $X_1,\, X_2$ and $X_3$, where each $X_i$ has three possible states with probabilities proportional to $(\theta_{i,1}^2, \theta_{i,1}\theta_{i,2}, \theta_{i,2}^2)$. For a generic scaling, the ML degree equals the degree of $X_{A,c}$, i.e.\ the normalized volume of the cube which is 48. However, for the all-ones scaling $c$, the ML degree of $X_{A,c}$ drops down to eight \cite[Theorem 5.3]{LikelihoodGeometryOfReflexivePolytopes}. Let $F$ be any of the six facets, let $A_F$ be the matrix with columns the lattice points of $F$, and let $c_F$ be the all-ones vector of length nine. Then we have $\mldeg(X_{A_F,c_F}) = 4 \leq 8$.
\end{example}

The paper is organized as follows. In \Cref{sec:preliminaries} we introduce notation and the relevant background on maximum likelihood estimation of toric varieties. \Cref{sec:monotonicity} is devoted to the proof of the main theorem stated above. The proof relies on choosing data with zero entries. Further implications of such data zeros are discussed in \Cref{sec:data_zeros}. A different approach to understanding the main result is taken in \Cref{sec:tropical}, where we analyze the ML degree monotonicity using tropical likelihood degenerations. In the final \Cref{sec:applications} we discuss applications of the main theorem to discrete graphical and quasi-independence models.

\begin{figure}[H]
\begin{minipage}[t]{0.49\textwidth}
(a)
\end{minipage}
\begin{minipage}[t]{0.49\textwidth}
(b)
\end{minipage}
\begin{minipage}[t]{0.49\textwidth}
\centering
\begin{tikzpicture}[scale=1.85,every node/.style={minimum size=1cm},on grid]
    \draw [black!40] (0,0) to (2,0);
    \draw [black!40] (0,0) to (0,2) to (1,2) to (1,0);
    \draw [black!40] (0.25,0.25) to (2.25,0.25);
    \draw [black!40] (0.25,0.25) to (0.25,2.25) to (1.25,2.25) to (1.25,0.25);
    \draw [black!40] (0,0) to (0.25,0.25);
    \draw [black!40] (1,0) to (1.25,0.25);
    \draw [black!40] (0,1) to (0.25,1.25) to (1.25,1.25) to (1,1) to (0,1);
    \draw [black!40] (0,2) to (0.25,2.25) to (1.25,2.25) to (1,2) to (0,2);
    \draw [black!40] (1,1) to (2,1) to (2.25,1.25) to (1.25,1.25);
    \draw [black!40] (1,2) to (2,2) to (2.25,2.25) to (1.25,2.25);
    \draw [black!40] (2,0) to (2.25,0.25) to (2.25,2.25);
    \draw [black!40] (2,0) to (2,2);
    \draw [black!40] (0.25,0.25) to (0.5,0.5) to (2.5,0.5) to (2.25,0.25);
    \draw [black!40] (1,0) to (1.5,0.5) to (1.5,2.5) to (1.25,2.25); 
    \draw [black!40] (1.25,1.25) to (1.5,1.5);
    \draw [black!40] (2.5,0.5) to (2.5,2.5) to (0.5,2.5) to (0.5,0.5);
    \draw [black!40] (0.25,1.25) to (0.5,1.5) to (2.5,1.5) to (2.25,1.25);
    \draw [black!40] (0.25,2.25) to (0.5,2.5);
    \draw [black!40] (2.25,2.25) to (2.5,2.5);
    
    \draw[black!80,fill=violet,opacity=0.5] (0,0) -- (2,0) -- (2.5,0.5) -- (0.5,0.5) -- (0,0); 
    \draw[black!80,fill=blue,opacity=0.45] (0,0) -- (0.5,0.5) -- (0.5,2.5) -- (0,2) -- (0,0); 
    \draw[black!80,fill=blue,opacity=0.45] (0.5,0.5) -- (0.5,2.5) -- (2.5,2.5) -- (2.5,0.5) -- (0.5,0.5); 
    
    \node [] at (0,-0.1) {\footnotesize1}; 
    \node [] at (1,-0.1) {\footnotesize1}; 
    \node [] at (2,-0.1) {\footnotesize1}; 
    \node [] at (0.19,0.31) {\footnotesize1}; 
    \node [] at (1.19,0.31) {\footnotesize1}; 
    \node [] at (2.3,0.31) {\footnotesize1}; 
    \node [] at (2.55,0.57) {\footnotesize1}; 
    \node [] at (1.55,0.57) {\footnotesize1}; 
    \node [] at (0.55,0.57) {\footnotesize1}; 
    \node [] at (0,1.1) {\footnotesize1};
    \node [] at (1,1.1) {\footnotesize1}; 
    \node [] at (2,1.1) {\footnotesize1}; 
    \node [] at (0.19,1.31) {\footnotesize1}; 
    \node [] at (1.19,1.31) {\footnotesize1}; 
    \node [] at (2.3,1.31) {\footnotesize1}; 
    \node [] at (2.55,1.57) {\footnotesize1}; 
    \node [] at (1.55,1.57) {\footnotesize1}; 
    \node [] at (0.55,1.57) {\footnotesize1}; 
    \node [] at (0,2.1) {\footnotesize1};
    \node [] at (1,2.1) {\footnotesize1}; 
    \node [] at (2,2.1) {\footnotesize1}; 
    \node [] at (0.19,2.31) {\footnotesize1}; 
    \node [] at (1.19,2.31) {\footnotesize1}; 
    \node [] at (2.3,2.31) {\footnotesize1}; 
    \node [] at (2.55,2.57) {\footnotesize1}; 
    \node [] at (1.55,2.57) {\footnotesize1}; 
    \node [] at (0.55,2.57) {\footnotesize1}; 
\end{tikzpicture}
\end{minipage}
\begin{minipage}[t]{0.49\textwidth}
\centering
\begin{tikzpicture}[scale=1.85,every node/.style={minimum size=1cm},on grid]
    \draw [black!30] (0,0) to (2,0);
    \draw [black!30] (0,0) to (0,2) to (1,2) to (1,0);
    \draw [black!30] (0.25,0.25) to (2.25,0.25);
    \draw [black!30] (0.25,0.25) to (0.25,2.25) to (1.25,2.25) to (1.25,0.25);
    \draw [black!30] (0,0) to (0.25,0.25);
    \draw [black!30] (1,0) to (1.25,0.25);
    \draw [black!30] (0,1) to (0.25,1.25) to (1.25,1.25) to (1,1) to (0,1);
    \draw [black!30] (0,2) to (0.25,2.25) to (1.25,2.25) to (1,2) to (0,2);
    \draw [black!30] (1,1) to (2,1) to (2.25,1.25) to (1.25,1.25);
    \draw [black!30] (1,2) to (2,2) to (2.25,2.25) to (1.25,2.25);
    \draw [black!30] (2,0) to (2.25,0.25) to (2.25,2.25);
    \draw [black!30] (2,0) to (2,2);
    \draw [black!30] (0.25,0.25) to (0.5,0.5) to (2.5,0.5) to (2.25,0.25);
    \draw [black!30] (1,0) to (1.5,0.5) to (1.5,2.5) to (1.25,2.25); 
    \draw [black!30] (1.25,1.25) to (1.5,1.5);
    \draw [black!30] (2.5,0.5) to (2.5,2.5) to (0.5,2.5) to (0.5,0.5);
    \draw [black!30] (0.25,1.25) to (0.5,1.5) to (2.5,1.5) to (2.25,1.25);
    \draw [black!30] (0.25,2.25) to (0.5,2.5);
    \draw [black!30] (2.25,2.25) to (2.5,2.5);
    
    \draw[black,fill=violet,opacity=0.5] (0,0) -- (2,0) -- (2.25,0.25) -- (1.5,0.5) -- (0.5,0.5) -- (0,0); 
    \draw[black,fill=blue,opacity=0.45] (0,0) -- (0,1) -- (0.5,0.5) -- (0,0); 
    \draw[black] (0,0) -- (0,1) -- (2,0) -- (0,0); 
    \draw[black] (2.5,0.5) -- (0,1) -- (2,0) -- (2.5,0.5); 
    \draw[black,fill=blue,opacity=0.45] (2.5,0.5) -- (0,1) -- (0.5,0.5) -- (2.5,0.5); 
    \draw[black,fill=violet,opacity=0.6] (2.25,0.25) -- (1.5,0.5) -- (2.5,0.5) -- (2.25,0.25); 
    
    \node [] at (0,-0.1) {\footnotesize1}; 
    \node [] at (1,-0.1) {\footnotesize2}; 
    \node [] at (2,-0.1) {\footnotesize1}; 
    \node [] at (0.19,0.31) {\footnotesize2}; 
    \node [] at (1.19,0.31) {\footnotesize4}; 
    \node [] at (2.25,0.31) {\footnotesize2}; 
    \node [] at (2.55,0.57) {\footnotesize1}; 
    \node [] at (1.55,0.57) {\footnotesize2}; 
    \node [] at (0.55,0.57) {\footnotesize1}; 
    \node [] at (0,1.1) {\footnotesize1};
\end{tikzpicture}
\end{minipage}

\caption{Examples of three-dimensional polytopes. In (a), the cube with side-length two is shown, while (b) illustrates the pyramid over the twice dilated unit square. Some faces or subpolytopes are highlighted by different colors. Each lattice point is labeled with a scaling.} \label{fig:polytopes}
\end{figure}

\section{ML Estimation of Log-Affine Models}
\label{sec:preliminaries}

Let $A\in \ZZ^{(d+1)\times n}$ be an integer matrix, where we fix the first row to be all-ones. Moreover, let $a_i\in\ZZ^d$ denote the $i^{\text{th}}$ column of $A$ with the leading one removed, i.e.\ $A$ is of the form
\[
    A = \begin{pmatrix}
        1 & 1 & \dots & 1\\
        a_1 & a_2 & \dots & a_n\\
    \end{pmatrix}.
\]
Let $c \in (\CC^*)^n$ be a vector of scalings. The data $(A,c)$ gives rise to a monomial map
\[
    \psi_{A,c} \colon (\CC^*)^{d+1} \rightarrow \PP^{n-1},\quad (\theta_0,\theta_1,\dots,\theta_d) \mapsto (c_1\theta_0\theta^{a_1} : \dots : c_n\theta_0\theta^{a_n}),
\]
where the notation $\theta^{a_j}$ is short for $\prod_{1\leq i\leq d}\theta_i^{a_{ij}}$. The \emph{scaled projective toric variety} $X_{A,c}$ is the Zariski closure of the image of $\psi_{A,c}$, i.e.\ $X_{A,c} = \overline{\im(\psi_{A,c})} \subseteq \PP^{n-1}$. We assume the matrix $A$ to be of full rank, the affine integer lattice generated by the columns to be equal to $\ZZ^{d+1}$, and the fibres of $\psi_{A,c}$ to be precisely $\CC^*$ (corresponding to the $s$ variable). 
% If $c=(1,\dots,1)$ is the all-ones vector we might drop $c$ from notation and simply write $X_{A,(1,\dots,1)} = X_A$.
A \emph{discrete statistical model} with $n$ outcomes is a subset $\M$ 
of the open \emph{probability simplex}
\[
    \Delta_{n-1}^{\circ}=\left\{(p_1,\ldots, p_n) ~:~ \sum_{i=1}^n p_i=1,~ p_i> 0 \textup{ for all } i=1,\dots,n\right\}.
\]%
A broad class of statistical models are \emph{log-affine models}, also called \emph{discrete regular exponential families}. These models are precisely the ones that arise from scaled toric varieties.

\begin{defprop}[{\cite[\S 6.2]{AlgebraicStatistics}}]
    A log-affine model $\M_{A,c}$ can be represented as the intersection of a scaled toric variety with the open probability simplex. More precisely, let $\varphi\colon\PP^{n-1}\rightarrow \CC^n$ be the map defined by $(p_1:\dots:p_n)\mapsto \frac{1}{p_1+\dots +p_n}(p_1,\dots,p_n)$. Then
    \begin{equation*}
        \label{eq:log-affine-model-as-toric-intersection}
        \M_{A,c} = \varphi(X_{A,c}) \cap \Delta^\circ_{n-1}.
    \end{equation*}
\end{defprop}

Given a data vector $u\in \RR^n_{>0}$ (usually containing the outcome counts of a random variable) and a model $\M$, a common task in statistics is to find the probability distribution $\hat{p}\in\M$ that best explains the data $u$. To this end, one tries maximizing the \emph{log-likelihood function}
\begin{equation}
    \label{eq:likelihoodFunction}
    \ell_u(p)= \left(\sum_{i=1}^n u_i\log(p_i)\right) - u_+\log(p_+).
\end{equation}
Here, we use the notation $u_+ = u_1+\dots+u_n$ and $p_+ = p_1+\dots+p_n$. A maximizer $\hat{p}$ of \eqref{eq:likelihoodFunction} on the model $\M$ is called \emph{maximum likelihood estimate (MLE)} of the model for data $u$. For log-affine models $\M_{A,c}$, there is a unique such maximizer, which is obtained as the intersection of a linear space with $X_{A,c}$. This is the content of \emph{Birch's Theorem}.

\begin{theorem}[Birch's Theorem, {\cite[Proposition 2.1.5]{LecturesOnAlgebraicStatistics}}]
    \label{thm:Birch}
    Let $u\in\NN^n_{>0}$ be a positive data vector and let $\M_{A,c}$ be a log-affine model. Then the MLE \, $\hat{p} \in \M_{A,c}$ for data $u$ is the unique nonnegative solution to the system of equations 
    \begin{equation}
        \label{eq:birch}
        A\hat{p} = (u_+)^{-1}Au \quad\text{and}\quad \hat{p}\in \varphi(X_{A,c}).
    \end{equation}
\end{theorem}

Moreover, the system \eqref{eq:birch} precisely describes all (complex) critical points of the log-likelihood function on $X_{A,c}$. Since we want the logarithm in \eqref{eq:likelihoodFunction} to be well-defined, we only look for critical points away from the \emph{distinguished hyperplane arrangement} $\cH = \V(p_1\dots p_np_+)$. This motivates the following definition, providing a critical point count.

\begin{definition}
    The \emph{maximum likelihood (ML) degree} $\mldeg(X_{A,c})$ of a scaled toric variety $X_{A,c}$ is the number of complex solutions for generic $u\in \CC^n$ to the system of equations
    \begin{equation}
        \label{eq:birchCritical}
        Ap = (u_+)^{-1}Au \quad\text{and}\quad p\in \varphi(X_{A,c}\setminus \cH).
    \end{equation}
\end{definition}

Since the system \eqref{eq:birchCritical} describes the intersection of an affine-linear space with the variety $X_{A,c}$, we immediately obtain the bound $\mldeg(X_{A,c}) \leq \deg(X_{A,c})$. By Kouchnirenko's Theorem, the degree is simply the normalized volume of the polytope $\conv(A)$. For generic scalings $c \in (\CC^*)^n$, this bound is attained \cite[Theorem 13]{amendola2019maximum}.

\section{Monotonicity of Toric ML Degree}
\label{sec:monotonicity}

For generic data, the ML degree monotonicity for facial submodels follows from well-known results in Ehrhart theory, in particular Stanley's monotonicity theorem \cite{AMonotonicityPropertyOfHVectorsAndHvectors}. More precisely, the ML degree for generic scalings equals the degree of the toric variety, which in turn is given by the normalized volume of $\conv(A)$. The normalized volume can be written as the sum of the $h^*$-vector components, and these are nonnegative and monotonic when considering subpolytopes. In this section, we prove our main theorem for all scalings, including non-generic ones. An important ingredient is the parameter continuation theorem \cite{CoefficientParameterPolynomialContinuation}.

\begin{theorem}[Parameter Continuation Theorem] \label{thm:parameter-continuation}
Let $F(\theta;u) : \CC^{d+1}\times\CC^n\to\CC^{d+1}$ be a square system of polynomials in $d+1$ variables and $n$ parameters. Furthermore, let $\cN(u)$ denote the number of isolated solutions, and let $\gamma(t):[0,1]\to\CC^n$ be a continuous path.
\begin{itemize}
\item[\textup{(i)}] $\cN(u)$ is finite and it is the same, say $\cN$, for almost all $u\in\CC^n$.
\item[\textup{(ii)}] For all $u\in\CC^n$, $\cN(u) \le \cN$.
\item[\textup{(iii)}] The subset of $\CC^n$  where $\cN(u)=\cN$ is a Zariski open set, denoted $U$.
\item[\textup{(iv)}]  If $\gamma([0,1]) \subseteq U$, then the homotopy $F(\theta,\gamma(t))$ defines $\cN$ continuous, isolated solution paths $\theta(t)\in\CC^{d+1}$.
\item[\textup{(v)}] If $\gamma((0,1]) \subseteq U$, then as $t \to 0$, the limits of the solution paths, if they exist, include all the isolated solutions to $F (\theta, \gamma(0)) = 0$.
\end{itemize}
\end{theorem}

A proof can be found in \cite{CoefficientParameterPolynomialContinuation}. We also refer to \cite{NumericalNonlinearAlgebra} and \cite[Chapter 7]{SommeseWampler} for more details on parameter homotopy continuation. As a well-known consequence, considering a square system, the number of isolated solutions can only decrease when deforming generic parameters to non-generic ones. Our main result is a powerful application of this fact.

\begin{theorem}[Monotonicity of toric ML degree]
    \label{thm:main}
    Let $X_{A,c}$ be a scaled toric variety and let $F$ be a face of $\conv(A)$. Then the ML degree of the facial submodel $X_{A_F,c_F}$ cannot increase,
    \begin{equation}
        \label{eq:monotonicity_ML_deg}
        \mldeg(X_{A_F,c_F}) \leq \mldeg(X_{A,c}).
    \end{equation}
\end{theorem}

\begin{proof}
We prove this by induction on the dimension of $X_{A,c}$. If $X_{A,c}$ is one-dimensional, all proper faces are vertices. The corresponding toric variety that defines a facial submodel is a single point which has ML degree one, so \eqref{eq:monotonicity_ML_deg} is trivially satisfied.
    
Now assume $\dim(X_{A,c}) = d > 1$. By induction, it suffices to assume that $F$ is a facet of $\conv(A)$. We apply an affine unimodular transformation so that $\conv(A)$ lies in the nonnegative orthant and $F$ lies in the hyperplane where the last coordinate is zero. Since affine unimodular transformations induce isomorphisms of very affine varieties, this leaves the ML degree unchanged, see \cite[\S 5 Proposition 1.2 \& \S 9 Proposition 1.4]{DiscriminantsResultantsAndMultidimensionalDeterminants}. Hence, upon reordering columns, we may assume that $A$ takes the form 
\[
	A = \begin{pmatrix} A_F & \star & \ldots & \star \\ 0 & a_{d (k+1)} & \ldots & a_{dn} \end{pmatrix} \in \mathbb{Z}^{(d+1)\times n},
\]
where $A_F$ is the $d\times k$ matrix corresponding to the facet $F$ (so $F$ contains $k$ lattice points). Setting $f = \sum_{i=1}^n c_i \theta^{a_i}$, by \cite[Definition 6]{amendola2019maximum}, the likelihood equations become
\[
	\cL_A(\theta;u) \coloneqq \left\{ \theta_0 f-1,\, \theta_0 \theta_1\partial_{\theta_1}f - u_+^{-1}(Au)_2,\, \ldots\, ,\, \theta_0 \theta_{d}\partial_{\theta_{d}}f - u_+^{-1}(Au)_{d+1} \right\}.
\]
The likelihood equations $\cL_{A_F}$ for the facial submodel defined by $X_{A_F,c_F}$ are obtained by replacing $A$ with the submatrix $A_F$. Note that for generic data $u=(u_1,\ldots,u_n)$ we may assume that $u_i\neq0$ for all $i=1,\ldots,n$. It follows that the corresponding critical points $p=(p_1,\ldots,p_n)$ of \eqref{eq:likelihoodFunction} have non-zero coordinates, see e.g.\ \cite[Proposition 4]{MaximumLikelihoodGeometryInThePresenceOfDataZeros}. Additionally, we intersect $X_{A,c}$ with the hyperplane where $p_1+\ldots+p_n=1$. Therefore, for generic $u$, we have $p \notin \cH$ and $\V(\cL_A(\theta;u))$ contains $\mldeg(X_{A,c})$ non-singular points. 

Now let $\alpha \coloneqq \min \{ a_{d(k+1)}, \ldots, a_{dn} \}-1\ge0$. We set $g \coloneqq \theta_d^{-\alpha}\partial_{\theta_d}f \in \CC[\theta_1, \ldots, \theta_d]$, and let $u_F = (u_1,\dots,u_k)$ be generic data for the facial subsystem. By construction of $g$, $\partial_{\theta_d}\theta_dg\vert_{\theta_d=0}$ is a nonzero polynomial in $\CC[\theta_1, \ldots, \theta_{d-1}]$.
% that does not depend on $u_F$. 
Its vanishing set has dimension $d-1$ while $\dim(X_{A_F,c_F})=d$. Hence, for generic $u_F$ we may assume that every solution $\hat\theta_F$ to $\cL_{A_F}(\theta_0,\dots,\theta_{d-1};u_F)=0$ is not in the vanishing set of $\partial_{\theta_d}\theta_dg\vert_{\theta_d=0}$. For such~$u_F$, $\V(\cL_{A_F}(\theta_0,\dots,\theta_{d-1};u_F))$ consists of $\mldeg(X_{A_F,c_F})$ many non-singular isolated points.

We extend the data vector $u_F$ by $n-k$ data zeros and set $\Tilde{u}=(u_F,0,\ldots,0)$. Any solution $\hat{\theta}_F \in \V(\cL_{A_F}(\theta_0,\dots,\theta_{d-1};u_F))$ lifts to a solution $\hat{\theta} = (\hat{\theta}_F, 0) \in \V(\cL_A(\theta;\Tilde{u}))$. It remains to show that $\hat{\theta}$ is an isolated point in $\V(\cL_A(\theta;\Tilde{u}))$. Then, by Theorem \ref{thm:parameter-continuation}, the number of isolated solutions can only decrease when specializing the data. Hence $\mldeg(X_{A_F,c_F}) \le \#\V(\cL_A(\theta;\Tilde{u})) \le \mldeg(X_{A,c})$, where $\#$ denotes the number of isolated points.

We decompose $\V(\cL_A(\theta;\Tilde{u}))$ into two sets $\V(\cL_1(\theta;\Tilde{u}))\cup\V(\cL_2(\theta;\Tilde{u}))$, where
\begin{align*}
\cL_1(\theta;\Tilde{u}) &\coloneqq \left \{ \theta_0 f-1,\, \theta_0 \theta_1\partial_{\theta_1}f - \tilde u_+^{-1}(A\tilde u)_2,\, \ldots\, ,\, \theta_0 \theta_{d-1}\partial_{\theta_{d-1}}f - \tilde u_+^{-1}(A\tilde u)_{d}, \theta_d^{\alpha+1} \right \}, \\
\cL_2(\theta;\Tilde{u}) &\coloneqq \left \{ \theta_0 f-1,\, \theta_0 \theta_1\partial_{\theta_1}f - \tilde u_+^{-1}(A\tilde u)_2,\, \ldots\, ,\, \theta_0 \theta_{d-1}\partial_{\theta_{d-1}}f - \tilde u_+^{-1}(A\tilde u)_{d}, \theta_0\theta_d g \right \}.
\end{align*}
Each point in $\V(\cL_1(\theta;\Tilde{u}))$ is a lifted point of an element in $\V(\cL_{A_F}(\theta_0,\dots,\theta_{d-1};u_F))$, i.e.\ $\hat{\theta}$ is isolated in $\V(\cL_1(\theta;\Tilde{u}))$. To show that $\hat{\theta}$ is also isolated in $\V(\cL_2(\theta;\Tilde{u}))$ we  compute the Jacobian $J_{\cL_2}(\theta)$ of $\cL_2(\theta;\Tilde{u})$. Considering the form of $A$, it follows that
\[
J_{\cL_2}(\theta)\vert_{\theta_d=0} = \begin{pmatrix} J_{\cL_{A_F}} & \star \\ 0 & \theta_0 \partial_{\theta_d}\theta_dg\vert_{\theta_d=0} \end{pmatrix},
\]
where $J_{\cL_{A_F}}$ is the Jacobian of $\cL_{A_F}(\theta_0,\dots,\theta_{d-1};u_F)$. With the choice of generic $u_F$ we obtain that both $J_{\cL_{A_F}}(\hat \theta_F)$ is full rank and $\theta_0 \partial_{\theta_d}\theta_dg\vert_{\theta_d=0} \neq 0$ for all $\hat\theta_F \in \V(\cL_{A_F}(\theta_0,\dots,\theta_{d-1};u_F))$. Thus, $\hat\theta$ is isolated in both $\V(\cL_1(\theta;\Tilde{u}))$ and $\V(\cL_2(\theta;\Tilde{u}))$ and therefore in $\V(\cL_A(\theta;\Tilde{u}))$.
\end{proof}

The following example shows that the ML degree monotonicity generally fails when considering arbitrary submodels that no longer correspond to faces of the original polytope.

\begin{example}
	Consider the square with side length two, located in the positive orthant with one vertex being the origin. Its lattice points define the following design matrix: $$A = \begin{pmatrix} 1 & 1 & 1 & 1 & 1 & 1 & 1 & 1 & 1 \\ 0 & 1 & 2 & 0 & 1 & 2 & 0 & 1 & 2 \\ 0 & 0 & 0 & 1 & 1 & 1 & 2 & 2 & 2 \end{pmatrix} \in \ZZ^{3\times9}.$$ 
    According to \cite[Theorem 5.9]{LikelihoodGeometryOfReflexivePolytopes}, we have $\mldeg(X_{A,c})=1$ for $c=(1,2,1,2,4,2,1,2,1)$. Denote $A'$ the submatrix of $A$ obtained by deleting the last column, and $c'$ the corresponding scaling. Then $\mldeg(X_{A',c'})=3$. Here, both $X_{A,c}$ and $X_{A',c'}$ have dimension two. Considering the pyramid over the square defines a toric variety of dimension three and ML degree one. The pyramid with corresponding ML degree one scaling is represented in Figure \ref{fig:polytopes} (b).
\end{example}

\section{Data Zeros}
\label{sec:data_zeros}

While maximum likelihood estimation is well understood with respect to generic data, the likelihood equations for non-generic data are largely unexplored. In particular, we often observe random incidents in which certain events do not occur \cite{PerformingTheExactTestOfHardyWeinbergProportionForMultipleAlleles}. In such cases, statisticians deal with discrete data sets that contain zero frequencies, and maximum likelihood estimation is performed for non-generic data. From an algebraic perspective, the occurrence of data zeros in the setting of toric models leads to interesting questions concerning the exceptional set of a parametrized polynomial system. For instance, continuous deformations of the data imply continuous deformations of the critical points \cite{TheMaximumLikelihoodDataSingularLocus}. When deforming generic data to specific data, the behavior of the critical points depends essentially on the choice of the specific data. For most choices, the corresponding critical points are distinct and regular as in the case of generic data. However, deformed critical points may coincide or lie in the singular locus. The data locus for which the critical points are singular was studied as the \emph{maximum likelihood data singular locus} by Horobe\textcommabelow{t} and Rodriguez \cite{TheMaximumLikelihoodDataSingularLocus}. Other references explicitly deal with the case of data zeros under the simplified assumption of sampling and model zeros \cite{MaximumLikelihoodGeometryInThePresenceOfDataZeros}. In particular, similar ML degree monotonicity statements are known for model zero varieties \cite[Corollary 12]{MaximumLikelihoodGeometryInThePresenceOfDataZeros}. An important observation is that occurring data zeros can cause deformed critical points that are singular but still isolated. However, data zeros outside of the support of a face can imply non-isolated critical points.

\begin{example}[Independence Model] \label{example:independenceMN3} \upshape
    We consider the independence model of two random variables  with state space $\{1,2,3\}$. After a unimodular transformation, the design matrix is
    $$
    A = \begin{pmatrix}
        1 & 1 & 1 & 1 & 1 & 1 & 1 & 1 & 1 \\
        1 & 1 & 1 & 0 & 0 & 0 & 0 & 0 & 0 \\
        0 & 0 & 0 & 1 & 1 & 1 & 0 & 0 & 0 \\
        1 & 0 & 0 & 1 & 0 & 0 & 1 & 0 & 0 \\
        0 & 1 & 0 & 0 & 1 & 0 & 0 & 1 & 0 \\
    \end{pmatrix}.
    $$
    While $\deg(X_A) = 6$, all ML degrees between one and six can be achieved by choosing appropriate scalings \cite{amendola2019maximum, MatroidStratificationOfMLDegreesOfIndependenceModels}. Examples of such scalings are shown in Table \ref{table:ScalingsMN3}. We denote these scaling matrices by $C_1, \ldots, C_6$, labeled by their corresponding ML degree. For data $(u_1, u_2, u_3, u_4, 0,0,0,0,0)$ and $(0,0,0,u_4,u_5,u_6,u_7,0,0)$, consider all possible ways of introducing further zeros into these data vectors, as shown in Table~\ref{table:DataZerosMN3}. For such data and the six scalings $C_1,\dots,C_6$, we computed the number of solutions to the likelihood equations symbolically using \texttt{Mathematica} and recorded them in Table~\ref{table:DataZerosMN3}.
\end{example}

\begin{table}
\centering
\begin{tabular}{ccccccc} \toprule
ML Degree  & 1 & 2 & 3 & 4 & 5 & 6 \\
\midrule
Scaling  & $\begin{bmatrix} 1 & 1 & 1 \\ 1 & 1 & 1 \\ 1 & 1 & 1 \end{bmatrix}$ & $\begin{bmatrix} 1 & 1 & 1 \\ 1 & 1 & 2 \\ 1 & 1 & 2 \end{bmatrix}$ & $\begin{bmatrix} 1 & 1 & 1 \\ 1 & 2 & 3 \\ 1 & 2 & 3 \end{bmatrix}$ & $\begin{bmatrix} 1 & 1 & 1 \\ 1 & 2 & 3 \\ 1 & 2 & 1 \end{bmatrix}$ & $\begin{bmatrix} 1 & 1 & 1 \\ 1 & 2 & 3 \\ 1 & 3 & 5 \end{bmatrix}$ & $\begin{bmatrix} 1 & 1 & 1 \\ 1 & 2 & 3 \\ 2 & 3 & 1 \end{bmatrix}$ \\
\bottomrule
\end{tabular}
\caption{Scalings for the model considered in Example \ref{example:independenceMN3} that cover all possible ML degrees.} \label{table:ScalingsMN3}
\end{table}

The above example illustrates that certain combinations of data zeros can lead to a non-transverse intersection of the scaled toric variety and the linear subspace $u+\ker(A)$. For other combinations, the number of solutions coincides with the ML degree, although we can only observe this behavior for the generic scaling. In principle, the number of isolated solutions can decrease. Since we did not find any reference for it, we record the following folklore statement as an immediate consequence of Theorem \ref{thm:parameter-continuation}.

\begin{corollary}
    Let $u$ be any data and $c \in (\CC^*)^n$. If $X_{A,c}$ and $u+\ker(A)$ intersect transversally, the number of complex solutions to the likelihood equations is at most $\mldeg(X_{A,c})$.
\end{corollary}

We believe this to be a good starting point for further discussions. While the exceptional locus for scalings is described by the principal $A$-determinant, there is also initial work on discriminants of the likelihood function in the data, see \cite{LogarithmicDiscriminantsOfHyperplaneArrangements} for a treatment in the case of linear models. Major open questions concern the computation of discriminants that describe the exceptional parameter locus in terms of both the data and the scalings. 

\begin{table}[h!]
\centering
\begin{tabular}{ |P{4.5cm}||P{1.45cm}|P{1.45cm}|P{1.45cm}|P{1.45cm}|P{1.45cm}|P{1.45cm}| }
\hline
Data & \multicolumn{6}{|c|}{Number of solutions to the likelihood equations for scalings} \\
\hline
& $C_1$ & $C_2$ & $C_3$ & $C_4$ & $C_5$ & $C_6$ \\
\hline \hline
$(u_1, u_2, u_3, u_4,0,0,0,0,0)$ & 0 & 0 & 0 & 1 & 2 & 3 \\
$(u_1, u_2, u_3, 0,0,0,0,0,0)$ & 0 & $\infty$ & $\infty$ & 1 & 2 & 3 \\
$(u_1, u_2, 0, u_4,0,0,0,0,0)$ & 0 & 0 & 0 & 1 & 1 & 2 \\
$(u_1, 0, u_3, u_4,0,0,0,0,0)$ & 0 & 0 & 0 & 1 & 2 & 3 \\
$(0, u_2, u_3, u_4,0,0,0,0,0)$ & 0 & 0 & 0 & 1 & 2 & 3 \\
$(u_1, u_2, 0, 0,0,0,0,0,0)$ & 0 & $\infty$ & $\infty$ & 0 & 1 & 2 \\
$(u_1, 0, u_3, 0,0,0,0,0,0)$ & 0 & $\infty$ & $\infty$ & 1 & 2 & 3 \\
$(0, u_2, u_3, 0,0,0,0,0,0)$ & 0 & $\infty$ & $\infty$ & 1 & 2 & 3 \\
$(u_1, 0, 0, u_4,0,0,0,0,0)$ & 0 & 0 & 0 & 1 & 2 & 2 \\
$(0, u_2, 0, u_4,0,0,0,0,0)$ & 0 & 0 & 0 & 0 & 1 & 2 \\
$(0, 0, u_3, u_4,0,0,0,0,0)$ & 0 & 0 & 0 & 1 & 2 & 3 \\
$(u_1, 0, 0, 0,0,0,0,0,0)$ & 0 & $\infty$ & $\infty$ & 0 & 1 & 2 \\
$(0, u_2, 0, 0,0,0,0,0,0)$ & 0 & $\infty$ & $\infty$ & $\infty$ & 1 & 2 \\
$(0, 0, u_3, 0,0,0,0,0,0)$ & 0 & $\infty$ & $\infty$ & 1 & 2 & 3 \\
$(0, 0, 0, u_4,0,0,0,0,0)$ & 0 & 0 & 0 & $\infty$ & 1 & 2 \\
\hline \hline
$(0, 0, 0, u_4, u_5, u_6, u_7,0,0)$ & 1 & 2 & 3 & 4 & 5 & 6 \\
$(0, 0, 0, u_4, u_5, u_6, 0,0,0)$ & 0 & 0 & 0 & 1 & 2 & 3 \\
$(0, 0, 0, u_4, u_5, 0, u_7,0,0)$ & 0 & 1 & 1 & 2 & 2 & 3 \\
$(0, 0, 0, u_4, 0, u_6, u_7,0,0)$ & 1 & 2 & 3 & 4 & 5 & 6 \\
$(0, 0, 0, 0, u_5, u_6, u_7,0,0)$ & 1 & 2 & 3 & 4 & 5 & 6 \\
$(0, 0, 0, u_4, u_5, 0, 0,0,0)$ & 0 & 0 & 0 & 1 & 1 & 2 \\
$(0, 0, 0, u_4, 0, u_6, 0,0,0)$ & 0 & 0 & 0 & $\infty$ & 2 & 3 \\
$(0, 0, 0, 0, u_5, u_6, 0,0,0)$ & 0 & 0 & 0 & 1 & 2 & 3 \\
$(0, 0, 0, u_4, 0, 0, u_7,0,0)$ & 0 & 1 & 1 & 2 & 2 & 3 \\
$(0, 0, 0, 0, u_5, 0, u_7,0,0)$ & 0 & 1 & 1 & 2 & 2 & 3 \\
$(0, 0, 0, 0, 0, u_6, u_7,0,0)$ & 1 & 2 & 3 & 4 & 5 & 6 \\
$(0, 0, 0, u_4, 0, 0, 0,0,0)$ & 0 & 0 & 0 & $\infty$ & 1 & 2 \\
$(0, 0, 0, 0, u_5, 0, 0,0,0)$ & 0 & 0 & 0 & 1 & 1 & 2 \\
$(0, 0, 0, 0, 0, u_6, 0,0,0)$ & 0 & 0 & 0 & $\infty$ & 2 & 3 \\
$(0, 0, 0, 0, 0, 0, u_7,0,0)$ & 0 & 1 & 1 & 2 & 2 & 3 \\
\hline
\end{tabular}
\caption{Computational results for the independence model of two random variables in the presence of data zeros. For certain combinations of data zeros, the table presents the number of solutions to the likelihood equations for all scalings given in Table \ref{table:ScalingsMN3}. The symbolic data entries are assumed to be generic.} \label{table:DataZerosMN3}
\end{table}

\section{Tropical Likelihood Degenerations}
\label{sec:tropical}

In this section, we introduce coefficients from the field of Puiseux series $\puiseux{\CC}$ into the polynomial $f$ and the data $u$. This leads to \emph{tropical maximum likelihood estimation}. In the articles \cite{agostini2023likelihood} and \cite{ boniface2024tropical} this has been studied for Puiseux data. Here, we also modify the polynomial $f$ to achieve the desired degeneration behavior.\par 
For generic valuations, there are still ML degree many solutions to the likelihood equations, now elements in $\puiseux{\CC}^{d+1}$. Taking the coordinatewise $t$-adic valuation of such solutions $\Tilde{\theta}$ gives rational vectors $\mathrm{val}(\Tilde{\theta})\in \QQ^{d+1}$, the tropical likelihood solutions. For toric varieties with generic scalings $c$, these tropical solutions have been characterized in \cite{boniface2024tropical}. In the generic case, Birch's Theorem can be tropicalized, implying that the tropical likelihood solutions are the stable intersection of a tropicalized toric variety (an affine-linear space) with a tropical linear space. For non-generic scalings, this characterization is no longer valid. \par 

We modify both the likelihood equations and the data in such a way that substituting $t=1$ gives the original likelihood equations for the toric variety $X_{A,c}$, while substituting $t=0$ results in the likelihood equations of $X_{A_F,c_F}$. This is achieved as follows. Let
\[
    \hat{f} = \sum_{a\in A} c_a t^{\delta_F(a)w_a}\theta^a,\quad \text{where}\quad \delta_F(a) = \left\{\begin{array}{ll}
        1 & \text{if}\quad a\notin F\\
        0 & \text{otherwise},
    \end{array}\right.
\]
and $w_a > 0$ are arbitrary positive weights. Moreover, let 
\[
    \hat{u} = (\hat{u}_a)_{a\in A} \quad \text{with}\quad \hat{u}_a = \left\{\begin{array}{ll}
        u_at^{w^\prime_a} & \text{if}\quad a\notin F\\
        u_a & \text{otherwise},
    \end{array}\right.
\]
with $w^\prime_a>0$. Then we define the set of tropical likelihood equations $\hat{\mathcal{L}}_A$ as
\begin{equation}
    \label{eq:tropical_likelihood}
    \hat{\mathcal{L}}_A = \left\{\theta_0\hat{f}-1,~ \theta_0 \theta_i\partial_{\theta_i}\hat{f} - \hat{u}_+^{-1}(A\hat{u})_i\quad\text{for } i=1,\dots,d \right\} \subset \puiseux{\CC}[\theta_0^{\pm},\theta_1^{\pm},\dots,\theta_d^{\pm}].
\end{equation}
Therefore, taking the $t\rightarrow 0$ limit is a \emph{tropical likelihood degeneration} \cite[\S6,7]{agostini2023likelihood}, turning the original likelihood solutions into the likelihood solutions of the facial submodel. The latter correspond to the tropical likelihood solutions of our modified equations. Hence, the degeneration allows to track which solutions persist when restricting to a face. We explain how to do this in the following example and provide code in \texttt{Maple}. We emphasize that, unlike in \cite{agostini2023likelihood}, we also introduce Puiseux coefficients in $f$, and not just in the data.
\begin{example}
    We consider the Segre embedding corresponding to the polytope $\mathrm{conv}((0,0),(1,0),(0,1),(1,1))$ and choose $\mathrm{conv}((0,0),(1,0))$ as our distinguished face $F$. For the polynomial $\hat{f}$, we set $\hat{f} = 1 + \theta_1 + 3t\theta_2 + 7t^3\theta_1\theta_2 \in \puiseux{\CC}[\theta_1,\theta_2]$. As tropical data we choose $\hat{u} = (1,2,3t^2,4t^4)$. Computing Gröbner bases for $\hat{\mathcal{L}}_A$ with different term orders leads to univariate polynomials of degree two in the variables $\theta_0,\theta_1$ and $\theta_2$ with coefficients in $\puiseux{\CC}$. For example, for the lexicographic order $\theta_0 \succ \theta_2 \succ \theta_1$ one obtains the polynomial
    \[
        (21t^4 + 7t^2)\theta_1^2 + (9t^4 - 14t^2 + 3)\theta_1 - 12t^4 - 6 \in \puiseux{\CC}[\theta_1].
    \]
    This polynomial has two roots in $\puiseux{\CC}$ which can be found via the Newton--Puiseux algorithm. In \texttt{Maple} this is implemented with the command \texttt{puiseux} from the \texttt{algcurves} package. This leads to the two solutions
    \begin{equation*}
        2 - 30t^4 +140t^6 + \frac{830}{3}t^8 +\dots,\quad -\frac{3}{7}t^{-2} + \frac{9}{7} - \frac{78}{7}t^2 + \frac{444}{7}t^4 + \dots 
    \end{equation*}
    Similarly, we obtain eight solution candidates for $(\theta_0,\theta_1,\theta_2)$. By plugging them back into $\hat{\mathcal{L}}_A$ we verify that there are two solutions to the tropical likelihood equations, namely\\
    \scalebox{0.83}{\parbox{.5\linewidth}{\begin{align*}
        & \left( -\frac{7}{3}t^2 - \frac{91}{9}t^4 + \frac{203}{27}t^6 + \frac{3745}{81} t^8 +\dots ,\,\, -\frac{3}{7}t^{-2} +\frac{9}{7} -\frac{78}{7}t^2 + \frac{444}{7}t^4 +\dots ,\,\, -\frac{1}{21}t^{-3} - \frac{2}{9}t^{-1} -\frac{10}{21}t +\frac{10}{3}t^3 + \dots \right) \\
        & \left(\frac{1}{3} -\frac{1}{3}t^2 + \frac{29}{9}t^4 - \frac{55}{3}t^6 + \dots ,\,\, 2 - 30t^4 + 140t^6 + \frac{830}{3}t^8 + \dots ,\,\, t - \frac{10}{3}t^3 + \frac{50}{9}t^5 + \frac{2090}{27}t^7+\dots \right).
    \end{align*}}}\\
    This agrees with the ML degree of this model being two. In the limit $t\rightarrow 0$, the first solution diverges. However, the second solution converges to $(1/3,2,0)$, which gives the unique solution of the likelihood equations for the facial submodel, which has ML degree one.
\end{example}
If we were to know that \eqref{eq:tropical_likelihood} forms a tropical basis, this would give an alternative proof of Theorem \ref{thm:main} via tropical lifting. Unfortunately, we are not aware of an argument proving that \eqref{eq:tropical_likelihood} is a tropical basis. A na\"ive application of the sufficient (but not necessary) criterion in \cite[Theorem 4.6.18]{maclagan2015introduction} can fail, as can be seen in the following example.

\begin{example}
    We consider the two-dilated cube as in Example \ref{ex:2-dilated-cube} with $F$ being the bottom face. Disregarding the homogenization, this yields three polynomials $\hat{\mathcal{L}}_A = \{\hat{f}_1,\hat{f}_2,\hat{f}_3\} \subset \puiseux{\CC}[\theta_1^{\pm},\dots,\theta_d^{\pm}]$, obtained by taking Euler derivatives of $\hat{f}$. For the weights $w_a$ and $w^\prime_a$ we pick random integers between 1 and 1000. These valuations give rise to weights on the lattice points of the Cayley configuration of the Newton polytopes of $\hat{f}_1,\hat{f}_2$ and $\hat{f}_3$. This induces a regular subdivision of the Cayley polytope. Theorem 4.6.18 in \cite{maclagan2015introduction} states that if this subdivision is a triangulation, then $\hat{\mathcal{L}}_A$ is a tropical basis. Using the command \texttt{subdivision\_of\_points} we are able to compute this subdivision in \texttt{Oscar.jl} \cite{OSCAR}. Computing \texttt{maximal\_cells} we then observe that the subdivision is not a triangulation. Of course, this does not exclude the possibility that for a better (non-random) choice of weights $w_a$ and $w^\prime_a$ one could obtain a triangulation of the Cayley polytope.
\end{example}

It would be desirable to make the tropical degeneration argument sketched here rigorous by finding a nice tropical basis for $\hat{\mathcal{L}}_A$, for a suitable choice of weights. Besides providing an alternative proof of Theorem \ref{thm:main}, this would refine Theorem \ref{thm:main} by describing the precise asymptotic behavior of the critical points in the limit $t\rightarrow 0$. A related analysis has been performed in \cite{sattelberger2023maximum}, where the authors study \emph{critical slopes} for schön varieties.

\begin{problem}
    Determine, for a suitable choice of weights $w_a$ and $w_a^\prime$ as above, a tropical basis for the toric likelihood equations $\hat{\mathcal{L}}_A$ with Puiseux coefficients.
\end{problem}

The authors express hope that it might be possible to apply Esterov's theory on engineered complete intersections \cite{esterov2024engineered} to make progress on this problem.

\section{Applications}
\label{sec:applications}

Our monotonicity result for the ML degree strengthens known results in the literature that had only used monotonicity in the very restrictive case when the ML degree of the full model is one. In this section, we present this for graphical and quasi-independence models.

Let $G$ be an undirected graph. \emph{Discrete graphical models} associated to $G$ are important examples of log-affine models. The matrix $A$ that defines the log-affine model $\mathcal{M}_G$ is obtained from a monomial parametrization based on the \textit{maximal cliques} of the graph $G$. We refer the reader to \cite{geiger2006toric} for details. 

When the graph is \emph{chordal}, that is, every induced cycle of length at least four contains a chord, then the ML degree is known to be one. To prove the converse, one needs to argue that non-chordal graphs have ML degree strictly larger than one. An important fact is that the $m$-cycle when $m\geq 4$ is non-chordal and indeed has ML degree larger than one. Then, by definition, a non-chordal graph contains an induced chordless $m$-cycle for some $m\geq 4$. Thus, if we knew the ML degree were monotonic, we would conclude that any non-chordal graph has also ML degree larger than one. The result that submodels of ML degree one models must be of ML degree one themselves is implicitly assumed in \cite[Theorem 4.4]{geiger2006toric}, and proved rigorously using Horn pair representations in \cite[Theorem 3.2]{QuasiIndependenceModelsWithRationalMaximumLikelihoodEstimator}.

Nevertheless, the arguments above do not exclude the possibility of having an induced cycle as a subgraph of a non-chordal graph $G$, such that the ML degree associated to the cycle is higher than the ML degree of the full model $\mathcal{M}_G$. As an application of Theorem \ref{thm:main}, we see that this cannot be the case.

\begin{corollary}\label{cor:discretegraph}
    Let $\mathcal{M}_G$ be an undirected graphical model for discrete variables and $H$ an induced subgraph of $G$. Then 
    $$\mldeg(\mathcal{M}_H) \leq \mldeg(\mathcal{M}_G). $$
\end{corollary}

\begin{proof}
    The key fact is that the polytope associated to $\mathcal{M}_H$ is a face of the polytope associated to $\mathcal{M}_G$, by \cite[Lemma A.2]{geiger2006toric}. Then Theorem \ref{thm:main} applies. 
\end{proof}

\begin{example}[Binary 4-cycle] \label{ex:binary4cycle}
We study the binary 4-cycle model, see \cite[Example 4]{geiger2006toric}. Since each random variable can take two values, the state space has cardinality $2^4=16$.  A full-rank matrix that represents the parametrization of this model is given by
$$ 
A = \begin{pmatrix}
1 & 1 & 1 & 1 & 1 & 1 & 1 & 1 & 1 & 1 & 1 & 1 & 1 & 1 & 1 & 1 \\
0 & 1 & 0 & 0 & 1 & 0 & 0 & 1 & 0 & 1 & 0 & 0 & 1 & 1 & 1 & 1 \\
0 & 0 & 1 & 0 & 0 & 0 & 1 & 1 & 1 & 1 & 0 & 1 & 0 & 1 & 0 & 1 \\
0 & 0 & 0 & 1 & 1 & 0 & 0 & 0 & 1 & 1 & 1 & 1 & 0 & 0 & 1 & 1 \\
0 & 0 & 0 & 0 & 0 & 1 & 1 & 0 & 0 & 0 & 1 & 1 & 1 & 1 & 1 & 1 \\
0 & 0 & 0 & 0 & 0 & 0 & 0 & 1 & 0 & 1 & 0 & 0 & 0 & 1 & 0 & 1 \\
0 & 0 & 0 & 0 & 0 & 0 & 0 & 0 & 1 & 1 & 0 & 1 & 0 & 0 & 0 & 1 \\
0 & 0 & 0 & 0 & 0 & 0 & 0 & 0 & 0 & 0 & 1 & 1 & 0 & 0 & 1 & 1 \\
0 & 0 & 0 & 0 & 0 & 0 & 0 & 0 & 0 & 0 & 0 & 0 & 1 & 1 & 1 & 1
\end{pmatrix}.
$$
There are 16 facets with ML degree five and eight facets with ML degree one. We choose a particular flag and compute the ML degrees for each face, with the results shown in Table \ref{table:flag}. As expected from Theorem \ref{thm:main}, the sequence of ML degrees of each face in the flag is weakly increasing with the dimension.
\end{example}

\begin{table}[H]
    \centering
    \begin{tabular}{clcc}
        \toprule
        Dimension & Face & ML Degree & Degree \\
        \midrule
        8 & $\text{conv}(a_1, \ldots, a_{16})$ & 13 & 64 \\ 
	7 & $\text{conv}(a_1, \ldots, a_{12})$ & 5 & 15 \\ 
	6 & $\text{conv}(a_1, \ldots, a_{10})$ & 3 & 8 \\ 
	5 & $\text{conv}(a_1, \ldots, a_8)$ & 2 & 4 \\
	4 & $\text{conv}(a_1, \ldots, a_7)$ & 2 & 4 \\ 
	3 & $\text{conv}(a_1, \ldots, a_5)$ & 1 & 2 \\
	2 & $\text{conv}(a_1, a_2, a_3)$ & 1 & 1 \\
	1 & $\text{conv}(a_1, a_2)$ & 1 & 1 \\
	0 & $\text{conv}(a_1)$ & 1 & 1 \\
        \bottomrule
    \end{tabular}
    \caption{ML degrees of a flag of $\conv(A)$, where $A$ is the matrix considered in Example~\ref{ex:binary4cycle}.}\label{table:flag}
\end{table}

As another application, we consider \emph{quasi-independence models}, 
where the original inspiration of the conjecture by Coons and Sullivant originated from \cite{QuasiIndependenceModelsWithRationalMaximumLikelihoodEstimator}. These models capture the independence of two discrete variables $X$ and $Y$ with $m$ and $k$ states, respectively, but some combinations of states cannot occur. More concretely, the variables can only take values in a subset $S \subseteq [m] \times [k]$ (the rest of the pairs are known as \textit{structural zeros}), but otherwise they are independent. This means that we have a monomial parametrization
$$\psi^S: \CC^{m+k} \rightarrow \CC^S, \quad (s,t) \mapsto s_it_j$$
where the coordinates of $\CC^S$ are indexed by $S$. The corresponding toric model is denoted by $\mathcal{M}_S$. Furthermore, to each such set $S \subseteq [m] \times [k]$, one associates naturally a bipartite graph $G_S$ defined on independent sets $[m]$ and $[k]$ with an edge between $i$ and $j$ if and only if $(i,j)\in S$. The main result of Coons and Sullivant is the following.

\begin{theorem}\cite[Theorem 1.3]{QuasiIndependenceModelsWithRationalMaximumLikelihoodEstimator}
    Let $S \subseteq [m] \times [k]$, let $\mathcal{M}_S$ be the associated quasi-independence model and let $G_S$ be the bipartite graph associated to $S$. Then $\mathcal{M}_S$ has maximum likelihood degree one if and only if $G_S$ is doubly chordal bipartite.
\end{theorem}

In order to prove the `only if' implication above, in their proof of \cite[Theorem 4.5]{QuasiIndependenceModelsWithRationalMaximumLikelihoodEstimator}, the authors use the weaker form of the statement that submodels of ML degree one models have themselves ML degree one \cite[Theorem 3.2]{QuasiIndependenceModelsWithRationalMaximumLikelihoodEstimator}.
The argument again leaves a priori the possibility of ML degree monotonicity being violated for quasi-independence models that do not have ML degree one. We rule this out thanks to Theorem \ref{thm:main}.

\begin{corollary}
    Let $T \subseteq S \subseteq [m] \times [k]$ such that $G_T$ is an induced subgraph of the bipartite graph $G_S$. Then the ML degrees of the associated quasi-independence models obey
    $$\mldeg(\mathcal{M}_T) \leq \mldeg(\mathcal{M}_S).$$
\end{corollary}

\begin{proof}
    Again, the key fact is that the induced subgraph $G_T$ of $G_S$ corresponds to a facial submodel $\mathcal{M}_T$ of the toric model $\mathcal{M}_S$, see \cite[Theorem 4.5]{QuasiIndependenceModelsWithRationalMaximumLikelihoodEstimator}. Thus Theorem \ref{thm:main} applies. 
\end{proof}

\bigskip
\textbf{Acknowledgements.} We thank Alexander Esterov, Bernd Sturmfels and Simon Telen for helpful discussions.

\begin{small}
\setlength{\itemsep}{-0.4mm}
\bibliographystyle{alpha}
\bibliography{bib.bib}
\end{small}

\bigskip \bigskip

\noindent
\footnotesize {\bf Authors' addresses:}

\noindent Carlos Am\'endola, TU Berlin \hfill \href{mailto:amendola@math.tu-berlin.de}{\tt amendola@math.tu-berlin.de}

\noindent Janike Oldekop, TU Berlin \hfill \href{mailto:oldekop@math.tu-berlin.de}{\tt oldekop@math.tu-berlin.de}

\noindent Maximilian Wiesmann,
CSBD Dresden \hfill \href{mailto:wiesmann@pks.mpg.de}{\tt wiesmann@pks.mpg.de}

\end{document}